\documentclass[english]{amsart}

\usepackage{esint}
\usepackage[svgnames]{xcolor} 
\usepackage[colorlinks,citecolor=red,pagebackref,hypertexnames=false,breaklinks]{hyperref}
\usepackage{pgf,tikz}
\usepackage{pdfsync}

\usepackage{dsfont}
\usepackage{url}
\usepackage[utf8]{inputenc}
\usepackage[T1]{fontenc}
\usepackage{lmodern}
\usepackage{babel}
\usepackage{mathtools}  
\usepackage{amssymb}
\usepackage{lipsum}
\usepackage{mathrsfs}
\usepackage{color}
\usepackage[skip=2.2pt plus 1pt, indent=12pt]{parskip}

\newtheorem{theorem}{Theorem}[section]
\newtheorem{proposition}{Proposition}[section]
\newtheorem{lemma}{Lemma}[section]

\numberwithin{equation}{section}

\title[Stability inequality]{Stability inequality for the problem of determining an unbounded potential from boundary measurements}

\author[Choulli]{Mourad Choulli}
\address{Universit\'e de Lorraine, Nancy, France}
\email{mourad.choulli@univ-lorraine.fr}

\date{}
 
\subjclass[2010]{35R30}
\keywords{Stability inequality, Schr\"odinger equation, unbounded potential, boundary measurements.}

\begin{document}

\begin{abstract}
We establish in dimension $3$ a stability inequality for the problem of determining the potential in the Schr\"odinger equation from boundary measurements in the case where the potential belongs to $L^s$ with $s\in (2,3)$.
\end{abstract}

\maketitle


\section{Introduction}

Let $D$ be a $C^{0,1}$ bounded domain of $\mathbb{R}^n$, $n\ge 3$.  Pick $q\in L^{n/2}(D)$ such that $0$ is not an eigenvalue of the operator $\Delta +q$ under Dirichlet boundary condition. According to \cite[Theorem 1.2]{Ch2} and the comments that follow it, we define the Dirichlet-to-Neumann map $\Sigma_q \in \mathscr{B}(H^{1/2}(\partial D),H^{-1/2}(\partial D))$ as follows
\[
\Sigma_q (f)=\partial_\nu u (f)\in H^{-1/2}(\partial D),
\]
where $u\in H^1(D)$ is the unique weak solution of the BVP
\[
(\Delta +q)u=0\quad \mbox{in}\; D,\quad u{_{|\partial D}}=f.
\]
This means that $u_{|\partial D}=f$ in the trace sense and
\[
-\int_D\nabla u\cdot \nabla vdx+\int quvdx=0,\quad v\in H_0^1(D).
\]
It is reported in \cite{DKS} that Lavine and Nachman proved that $\Sigma_q$ determines uniquely $q$. This result was announced in \cite{Na} but has never been published. Chanillo \cite[Theorem]{Cha} established the same uniqueness result for $q$ in a Fefferman-Phong class (containing $L^{n/2}$) with a small norm in this class. The uniqueness in the case $q\in L^s(D)$, $s>n/2$ is due to Jerison and Kenig (see the comments after \cite[Theorem]{Cha}). The uniqueness in the case where $\Delta$ is the Laplace-Beltrami operator on an admissible manifold $M$ and $q\in L^{n/2}(M)$ was obtained by Dos Santos Ferreira, Kenig and Salo \cite{DKS}.

We address the question of whether one can quantify this type of uniqueness results. When $q\in L^n(D)$, a logarithmic stability inequality can be established by modifying the proof in the case $q\in L^\infty(D)$ (e.g. \cite{Ch2}). Apart from this case, this question remains an open problem. We prove in the present work a stability result in dimension $3$ for $q\in L^s(D)$ with $s\in (2,3)$. Let us specify that our analysis does not allow us to include the case $s\in [3/2,2]$.

Let $Q'$ be a bounded rectangle of $\mathbb{R}^2$, $I=(1/2,3/2)$,  $\Omega \Subset Q=I\times Q'$ and we assume that $\Omega$ has $C^{0,1}$ regularity. The boundary of $\Omega$ will be denoted by $\Gamma$.

We choose $I=(1/2,3/2)$ for simplicity. However, the reader can verify that all results still hold if $I$ is replaced by any bounded interval of $\mathbb{R}$.
 
Throughout this text, we write $x=(x_1,x')\in I\times Q'$. The $L^r$-norm, $r\ge 1$, (resp. $H^1$-norm) is simply denoted by $\|\cdot \|_r$ (resp. $\|\cdot \|_{(1,2)}$). $H_0^1(\Omega)$ will be endowed with the norm $\|\nabla  \cdot \|_2$.

Recall that the natural norm of $L^{r_1}(I,L^{r_2}(Q'))$, $1\le r_1,r_2\le \infty$, is given by
\[
\|f\|_{r_1,r_2}=\left(\int_I\|f(x_1,\cdot)\|_{r_2}^{r_1}dx_1\right)^{1/r_1},\quad f\in L^{r_1}(I,L^{r_2}(Q')),\; r_1<\infty,
\]
and
\[
\|f\|_{\infty,r_2}=\sup_{x_1\in I}\|f(x_1,\cdot)\|_{r_2},\quad f\in L^{\infty}(I,L^{r_2}(Q')).
\]

The scalar product of $L^2(Q')$ will be denoted by $(\cdot |\cdot)$.

Let $s>2$, fix $\varkappa>0$, $t_0>0$, $\beta>0$ and $K>0$. Define $\mathscr{Q}$  as the set of functions $q\in L^{\infty,s}(Q)$ so that $\mathrm{supp}(q)\subset \Omega$, satisfy $\|q\|_{\infty,s}\le K$ and  for any $t\ge t_0$ it holds
\begin{equation}\label{0}
\|q\chi_{\{|q|>t\}}\|_{\infty,s}\le  \varkappa t^{-\beta}.
\end{equation}

Let us provide an example of an unbounded function belonging to $\mathscr{Q}$. Pick $x'_0\in Q'$, $q\in C(\overline{Q'}\setminus\{x'_0\})$ and assume that there exists $\delta>0$ so that 

\[
\lim_{x'\rightarrow x'_0}\frac{|q(x')|}{|x'-x'_0|^{-\delta}}<\infty.
\]
Switching to polar coordinates in a ball around $x'_0$, we obtain $q\in L^s(Q')$ if $\delta<2/s$. Furthermore, $q$ satisfies an estimate of the form \eqref{0} with $\beta=2/\delta -s$.

Our goal is to prove the following theorem. Hereinafter, $\|\Sigma_{q_1}-\Sigma_{q_2}\|$ denotes the norm of $\Sigma_{q_1}-\Sigma_{q_2}$ in $\mathscr{B}(H^{1 /2 }(\partial D),H^{-1/2}(\partial D))$, and $\tilde{Q}=\tilde{I}\times \tilde{Q}'$, where $\tilde{I}=(0,2)$ and $\tilde{Q}'\Supset Q'$ is a rectangle.

\begin{theorem}\label{theorem0}
Pick $0<\sigma \le 1$. There exist $C=C(Q,\tilde{Q},K,s,\sigma, \varkappa, t_0,\beta)>0$, $c=c(Q,K,\beta,t_0)>0$ and $\lambda_0=\lambda_0(Q,K,t_0)\ge t_0$ such that, for any $q_1,q_2\in \mathscr{Q}$ with $q=q_1-q_2$ independent of $x_1$, we have
\[
C\|q\|_{H^{-\sigma}(Q')}\le e^{c\lambda}\|\Sigma_{q_1}-\Sigma_{q_2}\|+\lambda^{-\tau},\quad \lambda \ge \lambda_0.
\] 
Here $\tau=\sigma\min(1/8,\beta/4)$.
\end{theorem}

We note that the inequality of theorem \ref{theorem0} gives a stability inequality with a logarithmic continuity modulus.

Our proof of theorem \ref{theorem0} follows the usual method based on the CGO solutions that we construct in section 2. While section 3 is devoted to the proof of theorem \ref{theorem0}. Some ideas we used to prove theorem \ref{theorem0} were borrowed from \cite{DKS}.

\section{Building CGO solutions}

In the following, we use the notation $\Delta=\partial_{x_1}^2+\Delta_{x'}$. Define the operator $A'$ by
\[
D(A')=\{u\in H_0^1(Q');  \Delta_{x'} u\in L^2(Q')\}
\]
and $A'u=-\Delta_{x'} u$ for $u\in D(A')$. Then set
\[
\Lambda=\{\lambda \in \mathbb{R};\; |\lambda|\ge 4;\; \lambda^2 \not\in \sigma(A')\},
\]
where $\sigma(A')$ stands for the (discret) spectrum of $A'$.

Pick $0<\alpha <1/2$, $2<r <2/\alpha -2$ and let $p'=2+\alpha(r-2)\in (2,4(1-\alpha))$ be the conjugate of some $1<p<2$. Set $\theta=\alpha r/p'$ $(<1/2)$ and define $1<\tilde{p}<2$ so that its conjugate $\tilde{p}'=1/\theta$.

\begin{proposition}\label{proposition1}
For every $\lambda \in \Lambda$, there exists $E_\lambda :L^2(Q)\rightarrow H^2(Q)$ such that
\begin{align*}
&e^{\lambda x_1}(-\Delta)e^{-\lambda x_1}E_\lambda f=f,\quad f\in L^2(Q),
\\
&E_\lambda e^{\lambda x_1}(-\Delta)e^{-\lambda x_1} f=f,\quad f\in C_0^\infty (Q).
\end{align*}
Furthermore, the following inequalities hold for any $f\in L^2(Q)$
\begin{align}
&\|E_\lambda f\|_2\le C_0|\lambda|^{-1}\|f\|_2,\label{est1}
\\
&\|E_\lambda f\|_{(1,2)}\le C_0\|f\|_2,\label{est2}
\\
&\|E_\lambda f\|_{\tilde{p}',p'}\le C\|f\|_{\tilde{p},p},\label{est3}
\end{align}
where $C_0=C_0(Q)>0$ and $C=C(Q,\alpha)>0$ are constants.
\end{proposition}

\begin{proof}
Let $(\lambda_j)_{j\ge 0}$ be the nondecreasing sequence of eigenvalues of the operator $A'$. Let $(\phi_j)_{j\ge 0}$ be an orthonormal basis of $L^2(Q')$ consisting of eigenfunctions with $A'\phi_j=\lambda_j \phi_j$ for each $j\ge 0$.

For each $k\in \mathbb{N}_0:=\mathbb{N}\cup \{0\}$, define the operator $\chi_k:L^2(Q') \rightarrow L^2(Q')$ as follows
\[
\chi_k u=\sum_{\sqrt{\lambda_j}\in [k,k+1)}(u|\phi_j)\phi_j,\quad u\in L^2(Q'),
\]
if $\{j\in \mathbb{N}_0;\; \sqrt{\lambda_j}\in [k,k+1)\}\ne \emptyset $, and $\chi_k=0$ if $\{j\in \mathbb{N}_0;\; \sqrt{\lambda_j}\in [k,k+1)\}= \emptyset $. 

Pick $k\in \mathbb{N}_0$. We check that
\begin{equation}\label{e1}
\|\chi_k u\|_2\le \|u\|_2, \quad u\in L^2(Q').
\end{equation}
On the other hand, as $\|\nabla \phi_j\|_2=\sqrt{\lambda_j}$ for each $j$, we find
\[
\|\nabla \chi_k u\|_2\le (k+1)\|u\|_2, \quad u\in L^2(Q').
\]
This and the fact that $H_0^1(Q')$ is continuously embedded in $L^r(Q')$ give
\begin{equation}\label{e2}
\|\chi_k u\|_r\le C_0(k+1)\|u\|_2, \quad u\in L^2(Q').
\end{equation}
Here and further $C_0=C_0(Q')>0$ is a generic constant. 

Using \eqref{e1}, \eqref{e2} and an interpolation inequality, we get 
\begin{equation}\label{e3}
\|\chi_k u\|_{p'}\le C_0(k+1)^\theta\|u\|_2, \quad u\in L^2(Q').
\end{equation}
Since $\chi_k$ is self-adjoint, we obtain
\[
\|\chi_ku\|_2^2=(\chi_k^2u|u)\le \|\chi_k^2u\|_{p'}\|u\|_{p},\quad u\in L^2(Q').
\]
This inequality and \eqref{e3} imply
\[
\|\chi_ku\|_2^2\le C_0(k+1)^\theta\|\chi_ku\|_2\|u\|_{p},\quad u\in L^2(Q').
\]
That is we have 
\begin{equation}\label{e4}
\|\chi_ku\|_2\le C_0(k+1)^\theta\|u\|_{p},\quad u\in L^2(Q').
\end{equation}
Pick $f\in L^2(Q)$. For simplicity, we identify $f$ with $f\chi_{Q}$ considered as a function defined on $\mathbb{R}\times Q'$.  Similarly to  \cite{DKS} or \cite{KSU}, we show that $E_\lambda$, $\lambda \in \Lambda$, is explicitly given by
\[
E_\lambda f(x_1,\cdot) =\psi(x_1)\sum_{j\ge 0}\int_{\mathbb{R}}h(\lambda ,\sqrt{\lambda_k} ,x_1-y_1)(f(y_1,\cdot)|\phi_j)\phi_jdy_1, \quad x_1\in \mathbb{R},
\]
where 
\[
h(\lambda,\mu,t)=\int_{\mathbb{R}}\frac{e^{its}}{s^2+2is\lambda +\mu^2-\lambda^2}ds,\quad t\in \mathbb{R},\; \mu >0,
\]
and $\psi \in C_0^\infty (\mathbb{R})$ satisfies $\psi=1$ in a neighborhood of $I$. Furthermore, \eqref{est1} and \eqref{est2} hold. It remains to prove \eqref{est3}.  Since $\chi_k^2=\chi_k$ and
\[
\|E_\lambda f(x_1,\cdot)\|_{p'}\le \sum_{k\ge 0}\|\chi_k^2 E_\lambda f(x_1,\cdot)\|_{p'},
\]
\eqref{e3} yields
\begin{equation}\label{ee1}
\|E_\lambda f(x_1,\cdot)\|_{p'}\le C_0\sum_{k\ge 0}(k+1)^{\theta}\|\chi_k E_\lambda f(x_1,\cdot)\|_2.
\end{equation}
On the other hand, using Minkowski's inequality, we get for $x_1\in I$
\begin{align*}
&\|\chi_k E_\lambda f(x_1,\cdot)\|_2\le \left(\sum_{\sqrt{\lambda_j}\in [k,k+1)}\int_{\mathbb{R}}|h(\lambda ,\sqrt{\lambda_j} ,x_1-y_1)|^2|(f(y_1,\cdot)|\phi_j)|^2dy_1\right)^{1/2}
\\
&\hskip 1cm\le \int_{\mathbb{R}} \left(\sum_{\sqrt{\lambda_j}\in [k,k+1)}|h(\lambda ,\sqrt{\lambda_j} ,x_1-y_1)|^2|( f(y_1,\cdot)|\phi_j)|^2dy_1\right)^{1/2}
\\
&\hskip 1cm\le \int_{\mathbb{R}} \left(\max_{\sqrt{\lambda_j}\in [k,k+1)} |h(\lambda ,\sqrt{\lambda_j} ,x_1-y_1)|^2\sum_{\sqrt{\lambda_j}\in [k,k+1)}|( f(y_1,\cdot)|\phi_j)|^2dy_1\right)^{1/2}
\\
&\hskip 1cm \le \int_{\mathbb{R}} \max_{\sqrt{\lambda_j}\in [k,k+1)} |h(\lambda ,\sqrt{\lambda_j} ,x_1-y_1)|\|\chi_kf(y_1,\cdot)\|_2dy_1
\end{align*}
which, in light of \eqref{e4}, gives
\[
\|\chi_k E_\lambda f(x_1,\cdot)\|_2\le C_0 (k+1)^\theta\int_{\mathbb{R}} \max_{\sqrt{\lambda_j}\in [k,k+1)} |h(\lambda ,\sqrt{\lambda_j} ,x_1-y_1)|\|f(y_1,\cdot)\|_pdy_1.
\]
This inequality in \eqref{ee1} implies
\begin{align}\label{ee2}
&\|E_\lambda f(x_1,\cdot)\|_{p'}
\\
&\qquad\le C_0\sum_{k\ge 0}(k+1)^{2\theta}\int_{\mathbb{R}} \max_{\sqrt{\lambda_j}\in [k,k+1)} |h(\lambda ,\sqrt{\lambda_j} ,x_1-y_1)|\|f(y_1,\cdot)\|_pdy_1.\nonumber
\end{align}
Let $t\in \mathbb{R}$. The following inequality is proved in \cite[Proposition 2.1]{DKS} 
\[
 \max_{\sqrt{\lambda_j}\in [k,k+1)} |h(\lambda ,\sqrt{\lambda_j} ,t)|\le \left\{
 \begin{array}{ll}
 k^{-1}e^{-(k-\lambda)|t|}\quad &\mbox{if}\; \lambda <k,
 \\
 k^{-1}  &\mbox{if}\; k\le \lambda <k+1
 \\
 k^{-1}e^{-(\lambda-(k+1))|t|} &\mbox{if}\; \lambda >k+1,
 \end{array}
 \right.
\]
and, from \cite[Lemma 2.3]{DKS}, we have
\[
|m(\lambda ,0,t)|\le |t|e^{-\lambda |t|}.
\]
We derive from these inequalities
\begin{align*}
&\sum_{k\ge 0}(k+1)^{2\theta} \max_{\sqrt{\lambda_j}\in [k,k+1)} |h(\lambda ,\sqrt{\lambda_j} ,t)|
\\
&\hskip 1cm\le e^{-\lambda |t|/2}+\sum_{1\le k\le \lambda-2}k^{2\theta-1}e^{-(\lambda-(k+1))|t|}+2\lambda^{2\theta-1}
\\
&\hskip 6cm  +\sum_{k> \lambda+1}k^{2\theta-1}e^{-(k-\lambda)|t|}.
\end{align*}
Hence
\begin{align*}
&\sum_{k\ge 0}(k+1)^{2\theta} \max_{\sqrt{\lambda_j}\in [k,k+1)} |h(\lambda ,\sqrt{\lambda_j} ,t)|
\\
& \hskip 2cm \le 3+\int_0^{\lambda-2}\rho^{2\theta-1}e^{((\lambda-2)-\rho))|t|}d\rho+\int_\lambda^\infty \rho^{2\theta-1}e^{(\rho-\lambda)|t|}d\rho
\\
& \hskip 2cm \le 3+|t|^{-2\theta}\left(\int_0^1\rho^{2\theta-1}d\rho+\int_1^\infty e^{-\rho}d\rho\right).
\\
& \hskip 2cm \le \kappa(1+|t|^{-2\theta}),
\end{align*}
where $\kappa=\kappa(\alpha)>0$ is a constant. Putting the last inequality in \eqref{ee2}, we find
\begin{align*}
\|E_\lambda f(x_1,\cdot)\|_{p'}
&\le C_0\kappa \int_{\mathbb{R}}(1+|x_1-y_1|^{-2\theta}) \|f(y_1,\cdot)\|_pdy_1
\\
&\le C_0\kappa \int_I\|f(y_1,\cdot)\|_pdy_1+C_0\kappa \int_{\mathbb{R}}|x_1-y_1|^{-2\theta} \|f(y_1,\cdot)\|_pdy_1.
\end{align*}
Since $2\theta=1-1/\tilde{p}+1/\tilde{p}'$, upon substituting $\kappa$ by a similar constant, applying  Hardy-Littlewood-Sobolev's inequality to the last term of the inequality above, we obtain
\[
\|E_\lambda f\|_{\tilde{p}',p'}\le C_0\kappa \|f\|_{\tilde{p},p}.
\]
This completes the proof.
\end{proof}

For each $r>2$, define $s_r=2r/(r-2)$, set $\ell=s_{p'}$ and $\tilde{\ell}=s_{\tilde{p}'}$. For simplicity, the norm of $\mathscr{B}(L^2(Q))$ will denoted also by $\|\cdot \|$.

\begin{lemma}\label{lemma1.0}
Let $h_j\in L^{\tilde{\ell},\ell}(Q)$, $j=0,1$. For any $\lambda \in \Lambda$  we have
\begin{equation}\label{e5.0}
\|h_0E_\lambda h_1\|\le C\|h_0\|_{\tilde{\ell},\ell}\|h_1\|_{\tilde{\ell},\ell},
\end{equation}
where $C$ is the constant in Proposition \ref{proposition1}. Furthermore, it holds 
\begin{equation}\label{e6.0}
\|h_0E_\lambda h_1\|\; \underset{|\lambda|\rightarrow \infty}{\longrightarrow}\; 0.
\end{equation}
\end{lemma}

\begin{proof}
Let $f\in L^2(Q)$. Since
\[
\|h_0E_\lambda h_1f\|_2^2=\int_I\|h_0(x_1,\cdot)E_\lambda h_1f(x_1,\cdot)\|_2^2dx_1,
\]
applying twice H\"older's inequality, we get
\begin{align*}
\|h_0E_\lambda h_1f\|_2^2&\le \int_I\|h_0(x_1,\cdot)\|_{\ell}^2\|E_\lambda h_1f(x_1,\cdot)\|_{p'}^2dx_1
\\
&\le \|h_0\|_{\tilde{\ell},\ell}^2\|E_\lambda h_1f\|_{\tilde{p}',p'}^2.
\end{align*}
Combined with \eqref{est3}, this inequality implies
\begin{equation}\label{e5.1}
\|h_0E_\lambda h_1f\|_2\le C \|h_0\|_{\tilde{\ell},\ell}\|h_1f\|_{\tilde{p},p}\le C\|h_0\|_{\tilde{\ell},\ell}\|h_1\|_{\tilde{\ell},\ell}\|f\|_2,
\end{equation}
from which we derive \eqref{e5.0}.

Let $0<\epsilon <1$ be fixed arbitrarily. We split $h_j$, $j=0,1$, into two terms $h_j=h_j^0+h_j^1$ with
\[
h_j^0=h_j\chi_{\{|h_j|\le \varrho\}},\quad h_j^1=h_j\chi_{\{|h_j|> \varrho\}}.
\]
We fix $\varrho>0$ sufficiently large in such a way that $\|h_j^1\|_{\tilde{\ell},\ell}\le \epsilon$, $j=0,1$. From \eqref{est1}, we have
\begin{equation}\label{e7.0}
\|h_0^0E_\lambda h_1^0\|\le C_0|\lambda|^{-1}\varrho^2.
\end{equation}
where $C_0=C_0(Q)>0$ is a constant. On the other hand, it follows from \eqref{e5.1} that
\begin{equation}\label{e8.0}
\|h_0^0E_\lambda h_1^1\|+ \|h_0^1E_\lambda h_1^0\| \le K_0\epsilon ,\quad \|h_0^1E_\lambda h_1^1\|\le \epsilon^2 ,
\end{equation}
where $K_0=\|h_0\|_{\tilde{\ell},\ell}+\|h_1\|_{\tilde{\ell},\ell}$. Choose $\lambda_\epsilon>0$ so that $|\lambda|^{-1}|\varrho^2\le \epsilon $ for each $\lambda \in \Lambda$ with $|\lambda| \ge \lambda_\epsilon$. 
Putting together \eqref{e7.0} and \eqref{e8.0}, we obtain
\[
\|h_0E_\lambda h_1\|\le (C_0+K_0+1)\epsilon,\quad \lambda \in \Lambda,\; |\lambda| \ge \lambda_\epsilon.
\]
That is we proved \eqref{e6.0}
\end{proof}

The notations
\[
N_j(\mu)=\|e^{\mu x_1}\|_{L^j(I)},\; \tilde{N}_j(\mu)=\|e^{\mu x_1}\|_{L^j(\tilde{I})}, \quad 1\le j\le \infty.
\]
will be used in the following.

\begin{proposition}\label{proposition2}
Let $q\in L^{\tilde{\ell}/2,\ell/2}(Q)$ satisfying $\|q\|_{\tilde{\ell}/2,\ell/2}\le K_0$ for some given $K_0>0$. There exist $\lambda_0=\lambda_0(Q,K_0)>0$ and a countable set  $\aleph$ so that $\Lambda_0:=\mathbb{R}\setminus \aleph \subset \Lambda$ and, for each $\xi \in \mathbb{S}^1$ and $\zeta=\lambda-\mu$ with $\lambda \in \Lambda_0$, $|\lambda|\ge \lambda_0$  and $\mu \in (-\infty ,0]$, we find  $u\in W^{2,p}(Q)\cap H^1(Q)$ a solution of the equation $(\Delta +q)u=0$ of the form
\begin{equation}\label{e16.0}
u=e^{\zeta (-x_1+ix'\cdot \xi)}+e^{-\lambda x_1}v. 
\end{equation} 
Furthermore, we have
\begin{equation}\label{e15.0.1}
\|u\|_{(1,2)}\le Ce^{2|\lambda|},
\end{equation}
where $C=C(Q,\tilde{Q},\alpha,K_0)>0$.
\end{proposition}

\begin{proof}
Let $\zeta=\lambda-\mu$ with $\lambda \in \Lambda$ and $\mu \in (-\infty ,0]$. Pick $\xi \in \mathbb{S}^{n-1}$. Then elementary calculations show that
\[
\Delta e^{\zeta (-x_1+ix'\cdot \xi)}=0.
\]
Let $q\in L^{\tilde{\ell}/2,\ell/2}(Q)$  satisfying $\|q\|_{\tilde{\ell}/2,\ell/2}\le K_0$. If $q=|q|e^{i\mathfrak{a}}$ we set $h_0=|q|^{1/2}e^{i\mathfrak{a}}$ and $h_1=|q|^{1/2}$. We have $h_j\in L^{\tilde{\ell},\ell}(Q)$, $j=0,1$, with
\begin{equation}\label{e11.0}
\|h_j\|_{\tilde{\ell},\ell}\le K_0^{1/2}.
\end{equation}
Also, $\phi=h_0e^{i\zeta x'\cdot\xi+\mu x_1}\in L^2(Q)$ with
\begin{equation}\label{e12.0}
\|\phi\|_2\le c_0N_\infty(\mu)\; (=c_0e^{-|\mu|/2}),
\end{equation}
where $c_0=c_0(Q,K_0)>0$ is a constant.

From Lemma \ref{lemma1.0} and its proof, there exits $\lambda_0=\lambda_0(Q,K_0)>0$ so that
\[
\|h_0E_\lambda h_1\|\le 1/2,\quad \lambda \in \Lambda,\; |\lambda|\ge \lambda_0.
\]
Fix $\lambda \in \Lambda$ arbitrarily so that  $|\lambda|\ge \lambda_0$. Define
\[
w=(1-h_0E_\lambda h_1)^{-1}\phi.
\]
Using \eqref{e12.0}, we find
\begin{equation}\label{e13.0}
\|w\|_2\le 2c_0N_\infty(\mu).
\end{equation}

Let $v=E_\lambda h_1w$. It follows from the third inequality in Proposition \ref{proposition1} that
\[
\|v\|_{\tilde{p}',p'}\le C_0\|h_1w\|_{\tilde{p},p}\le C_0\|h_1\|_{\tilde{\ell},\ell}\|\phi\|_2,
\]
where $C_0=C_0(Q,\alpha)>0$ is a constant. In light of \eqref{e11.0} and \eqref{e12.0}, we get from the last inequality
\begin{equation}\label{e14.0}
\|v\|_{\tilde{p}',p'}\le C_1N_\infty(\mu),
\end{equation}
where $C_1=C_1(Q,K_0,\alpha)$ is a constant.

Let 
\[
u=e^{\zeta (-x_1+ix'\cdot \xi)}+e^{-\lambda x_1}v. 
\]
We check that $(\Delta +q)u=0$ in $Q$ and, as $qu\in L^{\tilde{p},p}(Q)\subset L^p(Q)$ (simple calculations give $\tilde{p}-p=2(1-\alpha)/p'$), the interior $W^{2,p}$ regularity in $\tilde{Q}$ yields $u\in W^{2,p}(Q)$. Furthermore, applying \cite[Theorem 9.11, page 235]{GT}, we get
\begin{align*}
\|u\|_{W^{2,p}(Q)}&\le C'\left(\|u\|_{L^p(\tilde{Q})}+\|qu\|_{L^p(\tilde{Q})}\right)
\\
&\le C'\left(\|u\|_{L^2(\tilde{Q})}+\|qu\|_{L^{\tilde{p},p}(\tilde{Q})}\right)
\\
&\le C'\left(\|u\|_{L^2(\tilde{Q})}+\|q\|_{\tilde{\ell}/2,\ell/2}\|u\|_{L^{\tilde{p}',p'}(\tilde{Q})}\right)
\end{align*}
where $C'=C'(Q,\tilde{Q})>0$ is a constant. 

Next, since $p\ge \min_{0<\alpha<1/2}4(1-\alpha)/(3-4\alpha)=4/3\ge 5/6$, we derive that $W^{1,p}(Q)$ is continuously embedded in $L^2(Q)$. In consequence, $u\in H^1(Q)$ and
\[
\|u\|_{(1,2)}\le C'\left(\|u\|_{L^2(\tilde{Q})}+\|q\|_{\tilde{\ell}/2,\ell/2}\|u\|_{L^{\tilde{p}',p'}(\tilde{Q})}\right)
\]
Then \eqref{e15.0.1} follows readily from \eqref{e14.0} with $Q$ substituted by $\tilde{Q}$ (note that $\tilde{N}_\infty(\mu)=1$).
\end{proof}

In the remainder of this text, $\lambda_0$ and $\Lambda_0$ are as in Proposition \ref{proposition2} with $K_0=K$. By replacing $\lambda_0$ with $\max(\lambda_0,t_0)$, we assume that $\lambda_0\ge t_0$.

\begin{lemma}\label{lemma2}
Let  $q=|q|e^{i\mathfrak{a}}\in \mathscr{Q}$, $h_0=|q|^{1/2}e^{i\mathfrak{a}}$ and $h_1=|q|^{1/2}$. Let $v$ be defined as in the proof of Proposition \ref{proposition1}. Then the following inequality holds
\begin{equation}\label{e15}
\|v\|_2\le C|\lambda|^{-\gamma}N_\infty(\mu),\quad \mbox{for any}\; \lambda \in \Lambda_0,\; |\lambda |\ge \lambda_0,
\end{equation}
where $C=C(Q, K,\varkappa,t_0)$ and $\gamma=\min(1/2,\beta)$.
\end{lemma}

\begin{proof}
Pick $\lambda\in \Lambda_0$ with $|\lambda| \ge \lambda_0$. We split $h_0$ into two terms $h_0=h_0^0+h_0^1$ with
\[
h_0^0=h_j\chi_{\{|h_j|\le |\lambda|^{1/2}\}},\quad h_0^1=h_j\chi_{\{|h_j|> |\lambda|^{1/2}\}}.
\]
We show that the following inequalities hold
\[
\|h_0^0\|_\infty \le |\lambda|^{1/2}, \quad \|h_0^0\|_{\tilde{\ell},\ell}\le \|h_0\|_{\tilde{\ell},\ell}, \quad \|h_0^1\|_{\tilde{\ell},\ell}\le \varkappa |\lambda|^{-\beta}.
\]
In light of these inequalities, we get
\begin{align*}
\|v\|_2&\le \|E_\lambda h_0^0w\|_2+\|E_\lambda h_0^1w\|_2
\\
&\le C\left(|\lambda|^{-1/2}\|w\|_2+\|h_0^1w\|_{\tilde{p},p}\right)
\\
&\le C\left(|\lambda|^{-1/2}\|w\|_2+\|h_0^1\|_{\tilde{\ell},\ell}\|w\|_2\right)
\\
&\le C\left(|\lambda|^{-1/2}\|w\|_2+\varkappa|\lambda|^{-\beta}\|w\|_2\right)
\\
&\le C|\lambda|^{-\gamma}\|w\|_2,
\end{align*}
where $C=C(Q, K,\varkappa,t_0)$ is a constant. Combining this inequality with \eqref{e13.0}, we get \eqref{e15}.
\end{proof}

\section{Proof of Theorem \ref{theorem0}}

We will use the following lemma. 

\begin{lemma}\label{lemma3}
Let $s>2$. Then there exists $\alpha_s \in (0,1/2)$ so that $s=s_{p_s'}/2$, where $p'_s=2+\alpha_s(r_s-2)$ for some $r_s \in (2,2/\alpha_s-2)$.
\end{lemma}

\begin{proof}
As $\alpha \in (0,1/2)\mapsto (4-4\alpha)/(2-4\alpha)$ is an increasing function, $\lim_{t\uparrow 1/2}(4-4\alpha)/(2-4\alpha)=+\infty$ and $\lim_{t\downarrow 0}(4-4\alpha)/(2-4\alpha)=2$, there exists $\alpha_s\in (0,1/2)$ so that $(4-4\alpha_s)/(2-4\alpha_s)<s$. Since $t\in (2, 4(1-\alpha_s))\mapsto s_{t}/2\in ((4-4\alpha_s)/(2-4\alpha_s),\infty)$ is bijective, we find $p'_s\in (2, 4(1-\alpha_s))$ so that $s_{p'_s}/2=s$. Next, using that $r\in (2,2/\alpha_s-2)\mapsto p'(r)=2+\alpha (r-2)\in (2,4(1-\alpha_s))$ is also bijective in order to derive that there exists $r_s\in (2,2/\alpha_s-2)$ such that $p'(r_s)=p'_s$. 
\end{proof}

\begin{proof}[Proof of Theorem \ref{theorem0}]
Let $\alpha_s$, $r_s$ and $p'_s$ as in the preceding lemma so that $s=s_{p'}/2:=\ell/2$, where $p':=p'_s$. Define $\tilde{p}'=p'_s/(\alpha_sr_s)$ and set $\tilde{\ell}=s_{\tilde{p}'}$. In this case $q_j\in L^{\tilde{\ell}/2,\ell/2}(Q)$, $j=1,2$. Let $\lambda \in \Lambda_0$, $\mu \in (-\infty ,0]$ and $\xi \in \mathbb{S}^1$. We denote by $u_1$ the solution $u$ of \eqref{e16.0} when $q=q_1$ and we denote by $u_2$ the solution of \eqref{e16.0} when $q=q_2$, $\mu =0$ and $\lambda$ is replaced by $-\lambda$. That is we have
\begin{align*}
&u_1=e^{(\lambda -\mu) (-x_1+ix'\cdot \xi)}+e^{-\lambda x_1}v_1=e^{-\lambda x_1}(e^{ -\mu (-x_1+ix'\cdot \xi)+i\lambda x'\cdot \xi }+v_1),
\\
&u_2=e^{-\lambda (-x_1+ix'\cdot \xi)}+e^{\lambda x_1}v_2=e^{\lambda x_1}(e^{-i\lambda x'\cdot \xi}+v_2).
\end{align*}
Whence
\[
u_1u_2=e^{-\mu (-x_1+ix'\cdot \xi)}+\tilde{u},
\]
where
\[
\tilde{u}=e^{-i\lambda x\cdot \xi}v_1+e^{ -\mu (-x_1+ix'\cdot \xi)+i\lambda x\cdot\xi}v_2+v_1v_2.
\]
We extend $q=q(x')$ by $0$ outside $Q'$ and we still denote this extension by $q$. We have
\[
\int_\Omega qu_1u_2=\hat{q}(\mu \xi)\int_I e^{\mu x_1}dx_1+\int_\Omega q\tilde{u}dx,
\]
where $\hat{q}$ denotes the Fourier transform of $q$. In this identity, replacing $\xi$ by $-\xi$  we obtain
\[
\int_\Omega qu_1u_2=\hat{q}(|\mu| \xi)N_1(\mu)+\int_\Omega q\tilde{u}dx,
\]
or equivalently
\[
\int_\Omega qu_1u_2=\hat{q}(\eta) N_1(\mu)+\int_\Omega q\tilde{u}dx,\quad \eta \in \mathbb{R}^2,\; \mu=-|\eta|.
\]
Hence

\begin{equation}\label{e18.0}
N_1(\mu)|\hat{q}(\eta)|\le \left|\int_\Omega qu_1u_2\right| +\|q\tilde{u}\|_1,\quad \eta \in \mathbb{R}^2,\; \mu=-|\eta|.
\end{equation}

Next,  applying \cite[Lemma 1.4]{Ch2}, we obtain
\begin{equation}\label{e19.0}
\int_\Omega qu_1u_2dx=\langle (\Sigma_2-\Sigma_1)(u_1{_{|\Gamma}})|u_2{_{|\Gamma}}\rangle,
\end{equation}
where $\langle \cdot |\cdot\rangle $ is the duality pairing between $H^{1/2}(\Gamma)$ and $H^{-1/2}(\Gamma)$.

From \eqref{e18.0} and \eqref{e19.0}, we get
\begin{equation}\label{e18.1}
N_1(\mu)|\hat{q}(\eta)|\le \left|\langle (\Sigma_2-\Sigma_1)(u_1{_{|\Gamma}})|u_2{_{|\Gamma}}\rangle\right| +\|q\tilde{u}\|_1,\quad \eta \in \mathbb{R}^2,\; \mu=-|\eta|.
\end{equation}

Il light of \eqref{e15.0.1}, we obtain
\[
\|u_j\|_{(1,2)}\le Ce^{2|\lambda| },\quad j=1,2.
\]
Here and below $C=C(Q,\tilde{Q},K,\varkappa ,\beta,t_0)>0$ is generic constants. Therefore, we have
\begin{equation}\label{e18.2}
\left|\langle (\Sigma_2-\Sigma_1)(u_1{_{|\Gamma}})|u_2{_{|\Gamma}}\rangle\right| \le Ce^{2|\lambda |}\|\Sigma_2-\Sigma_1\|,
\end{equation}
where we used the continuity of the trace map $w\in H^1(\Omega)\mapsto w_{|\Gamma}\in H^{1/2}(\Gamma)$.

On the other hand, we have
\begin{align*}
\|q\tilde{u}\|_1&\le \|q\|_2(\|v_1\|_2+N_\infty (\mu)\|v\|_2)+\|q\|_{\tilde{s},s}\|v_1v_2\|_{\tilde{s}',s'}
\\
&\le \|q\|_2(\|v_1\|_2+N_\infty (\mu)\|v\|_2)+\|q\|_{\tilde{s},s}\|v_1\|_2\|v_2\|_{\tilde{p}',p'}
\end{align*}

Assume that $|\lambda|\ge \lambda_0$. Then we derive from \eqref{e14.0} and \eqref{e15} that
\begin{equation}\label{e17}
\|\tilde{u}\|_2\le C|\lambda|^{-\gamma}N_\infty(\mu),
\end{equation}
where $\gamma=\min(1/2,\beta)$. Inequalities \eqref{e18.2} and \eqref{e17} in \eqref{e18.1} yield
\begin{equation}\label{e18.3.0}
N_1(\mu)|\hat{q}(\eta)|\le Ce^{2|\lambda|} \|\Sigma_2-\Sigma_1\|+|\lambda|^{-\gamma}N_\infty(\mu), \quad \eta \in \mathbb{R}^2,\; \mu=-|\eta|.
\end{equation}
Simple calculations give
\[
N_1(\mu)\ge e^{-2|\eta|}\quad \mbox{and}\quad \frac{N_\infty(\mu)}{N_1(\mu)}\le (1+|\eta|)\quad \eta \in \mathbb{R}^2,\; \mu=-|\eta|.
\]
Therefore \eqref{e18.3.0} implies
\begin{equation}\label{e18.3}
C|\hat{q}(\eta)|\le Ce^{2(|\lambda|+|\eta|)} \|\Sigma_2-\Sigma_1\|+(1+|\eta|)|\lambda|^{-\gamma}, \quad \eta \in \mathbb{R}^2.
\end{equation}

Note that, since  the right hand side is continuous in $\lambda$, inequality \eqref{e18.3} holds for every $|\lambda|\ge \lambda_0$. In particular, we have
\begin{equation}\label{e18}
C|\hat{q}(\eta)|^2\le e^{4(\lambda +|\eta|)} \|\Sigma_2-\Sigma_1\|^2+(1+|\eta|)^2\lambda^{-2\gamma},\quad \lambda \ge \lambda_0.
\end{equation}

Fix $\sigma>0$. Then \eqref{e18} implies
\begin{equation}\label{e19}
C\int_{B_\rho}(1+|\eta|^2)^{-\sigma}|\hat{q}(\eta)|^2d\eta \le e^{5(\lambda+\rho)} \|\Sigma_2-\Sigma_1\|^2+\rho^2(1+\rho^2)|\lambda|^{-2\gamma}
\end{equation}
$\rho >0$, $\lambda \ge \lambda_0$, where $B_\rho=B(0,\rho)$. Also, we have
\[
\int_{\{|\eta|\ge \rho\}}(1+|\eta|^2)^{-\sigma}|\hat{q}(\eta)|^2d\xi\le \rho^{-2\sigma}\|\hat{q}\|_2=(2\pi)^2\rho^{-2\sigma}\|q\|_2\le (2\pi)^2K^2 \rho^{-2\sigma}
\]
This inequality together with \eqref{e19} give
\[
C\|q\chi_{Q'}\|_{H^{-\sigma}}\le e^{5(\lambda+\rho)} \|\Sigma_2-\Sigma_1\|+\rho(1+\rho)\lambda^{-\gamma}+\rho^{-\sigma},\quad \rho >0,\; \lambda \ge \lambda_0,
\]
and hence
\begin{equation}\label{e20}
C\|q\|_{H^{-\sigma}(Q')}\le e^{5(\lambda+\rho)/2} \|\Sigma_2-\Sigma_1\|+\rho (1+\rho)\lambda^{-\gamma}+\rho^{-\sigma}, \quad \rho >0,\; \lambda \ge \lambda_0.
\end{equation}

Taking $\rho =\lambda^{\gamma/4}$ in \eqref{e20}, we find
\[
C\|q\|_{H^{-\sigma}(Q')}\le e^{c_0\lambda}\|\Sigma_2-\Sigma_1\|+\lambda^{-\sigma\gamma/4},\quad \lambda\ge \lambda_0,
\]
where $c_0=c_0(Q,K,\beta,t_0)$. This is the expected inequality.
\end{proof}

\end{document}